\documentclass[a4paper,10pt]{amsart}
\usepackage[T1]{fontenc}
\usepackage[english]{babel}
\usepackage{amsmath,amsthm}
\usepackage{amsfonts}
\usepackage{comment}
\usepackage{enumerate}
\usepackage{graphicx}
\usepackage{hyperref}
\usepackage{subcaption}
\usepackage{color}
\usepackage[all]{xy}
\usepackage{thmtools, thm-restate}
\usepackage{overpic}
\usepackage{amssymb,tikz}
\usepackage{enumitem}
\usepackage{multicol}
\usepackage{alltt}

\usetikzlibrary{decorations.pathreplacing}

\usepackage[top=1.23in, bottom=1.22in, left=1.23in, right=1.23in]{geometry}

\newtheorem{thm}{Theorem}[section]
\newtheorem{cor}[thm]{Corollary}
\newtheorem{lem}[thm]{Lemma}

\theoremstyle{definition}
\newtheorem{defn}[thm]{Definition}

\newtheorem{ques}{Question}

\theoremstyle{remark}
\newtheorem{rem}[thm]{Remark}

\numberwithin{equation}{section}

\newcommand{\spin}{\ifmmode{\rm Spin}\else{${\rm spin}$\ }\fi}
\newcommand{\spinc}{\ifmmode{{\rm Spin}^c}\else{${\rm spin}^c$}\fi}

\newcommand{\Z}{\mathbb{Z}}

\newcommand{\plumb}{\entrymodifiers={+[o][F-]} \xymatrix@C=8pt}
\newcommand{\el}{\ar@{-}[r]}
\newcommand{\ed}{\ar@{..}[r] }

\begin{document}

\title{Minimum $b_2$ of symplectic fillings of lens spaces}%

\author{Antony T.H. Fung}%
\address {Seoul National University}
\email{thf25@cantab.ac.uk}

\date{\today}%

\begin{abstract}
We classify lens spaces for which the canonical negative-definite plumbing has minimal $b_2$ among all symplectic fillings of the standard contact structure. The classification is given by 4 infinite families of ``forbidden subgraphs'' in a manner similar to Aceto-McCoy-Park's work on smooth negative-definite fillings of lens spaces. In particular, no classification of this form can be given using a finite list of forbidden configurations, highlighting a difference between symplectic fillings and smooth negative-definite fillings of lens spaces.
\end{abstract}
\maketitle

\section{Introduction}

The interaction between 3-manifolds and 4-manifolds is a central topic in low-dimensional topology. Given a fixed 3-manifold $Y$, one often asks which 4-manifolds $X$ have boundary $Y$. Lens spaces form a class of 3-manifolds that has been extensively studied. There have been numerous studies of which 4-manifolds $X$, under various restrictions, are bounded by lens spaces. Some examples of such restrictions include: Smooth and definite \cite{definite}, being a smooth rational homology ball \cite{Lisca:2007-1}, smooth and simply connected \cite{JoParkPark_AGT}, symplectic \cite{Lisca:2008-1,bhupal2016symplectic,Fossati:2020-1,Etnyre-Roy:2021-1,ChristianLi}, symplectic and not simply connected \cite{Aceto-McCoy-Park:2020-1}, built from a 2-handle \cite{Greene:2013-1}, built from a 2-handle and a cobordism \cite{Fung2024}, and with restrictions coming from algebraic geometry \cite{JoParkPark_Trans,jo2026algebraicmontgomeryyangproblem}.

Throughout this paper, we use the convention that the lens space $L(p,q)$ is obtained by $-p\slash q$-surgery on the unknot in $S^3$. We write $\xi_{\mathrm{st}}$ to denote the standard contact structure on $L(p,q)$, which is induced by the standard contact structure on $S^3$ under the quotient description of $L(p,q)$. The lens space $L(p,q)$ bounds the canonical negative-definite plumbing $X(p,q)$ which is obtained by plumbing disk bundles over spheres along the weighted graph
$$\begin{tikzpicture}[xscale=1.0,yscale=1,baseline={(0,0)}]
    \node at (0.9,0.4) {$-a_1$};
    \node at (1.9,0.4) {$-a_2$};
    \node at (3.9,0.4) {$-a_k$};
    \node (A1) at (1, 0) {$\bullet$};
    \node (A2) at (2, 0) {$\bullet$};
    \node (A3) at (3, 0) {$\cdots$};
    \node (A4) at (4, 0) {$\bullet$};
    \path (A1) edge [-] node [auto] {$\scriptstyle{}$} (A2);
    \path (A2) edge [-] node [auto] {$\scriptstyle{}$} (A3);
    \path (A3) edge [-] node [auto] {$\scriptstyle{}$} (A4);
\end{tikzpicture},$$
where $a_i\ge 2$ comes from the Hirzebruch-Jung expansion of $p\slash q$:
$$\dfrac{p}{q}= a_1 - \cfrac{1}{a_2 - \cfrac{1}{\ddots - \cfrac{1}{a_k}}}.$$
The manifold $X(p,q)$ is a smooth negative-definite filling of $L(p,q)$. It admits a symplectic structure that fills $(L(p,q),\xi_{\mathrm{st}})$.

Recall that given a graph, an induced subgraph is a graph obtained by taking a subset of the vertices and all edges that connect pairs	of these chosen vertices. In \cite{definite}, Aceto-McCoy-Park studied which lens spaces (and their connected sums) satisfy the property that the canonical negative-definite plumbing has minimal second Betti number $b_2$ among all smooth negative-definite fillings. They concluded that this property is equivalent to the combinatorial statement that the plumbing graph above does not contain any of the 10 ``forbidden configurations'' they wrote down as an induced subgraph.

A natural question to ask is whether a similar classification exists in the symplectic category. Is there a list of configurations such that the canonical negative-definite plumbing having the minimal $b_2$ among all symplectic fillings of the lens space is equivalent to the plumbing graph not containing any of those configurations as an induced subgraph?

In this paper, we provide such classification in the symplectic category by explicitly writing down a list of configurations with this property. To state the main theorem of this paper, we first define the following families of weighted graphs:

\begin{align*}
A_k&:=\begin{tikzpicture}[xscale=1.0,yscale=1,baseline={(0,0)}]
    \node at (0.9,0.4) {$-(4+k)$};
    \node at (1.9,0.4) {$-2$};
		\node at (3.9,0.4) {$-2$};
    \node (A1) at (1,0) {$\bullet$};
    \node (A2) at (2,0) {$\bullet$};
		\node (A3) at (3,0) {$\cdots$};
		\node (A4) at (4,0) {$\bullet$};
    \path (A1) edge [-] node [auto] {$\scriptstyle{}$} (A2);
		\path (A2) edge [-] node [auto] {$\scriptstyle{}$} (A3);
		\path (A3) edge [-] node [auto] {$\scriptstyle{}$} (A4);
		\draw [decorate,decoration={brace,amplitude=5pt,mirror,raise=2ex}] (2,0) -- (4,0) node[midway,yshift=-2em]{k};
  \end{tikzpicture} \\
B_k&:=\begin{tikzpicture}[xscale=1.0,yscale=1,baseline={(0,0)}]
    \node at (0.9,0.4) {$-3$};
    \node at (1.9,0.4) {$-2$};
		\node at (3.9,0.4) {$-2$};
		\node at (4.9,0.4) {$-3$};
    \node (A1) at (1,0) {$\bullet$};
    \node (A2) at (2,0) {$\bullet$};
		\node (A3) at (3,0) {$\cdots$};
		\node (A4) at (4,0) {$\bullet$};
		\node (A5) at (5,0) {$\bullet$};
    \path (A1) edge [-] node [auto] {$\scriptstyle{}$} (A2);
		\path (A2) edge [-] node [auto] {$\scriptstyle{}$} (A3);
		\path (A3) edge [-] node [auto] {$\scriptstyle{}$} (A4);
		\path (A4) edge [-] node [auto] {$\scriptstyle{}$} (A5);
		\draw [decorate,decoration={brace,amplitude=5pt,mirror,raise=2ex}] (2,0) -- (4,0) node[midway,yshift=-2em]{$k$};
  \end{tikzpicture} \\
C_{k,s}&:=\begin{tikzpicture}[xscale=1.0,yscale=1,baseline={(0,0)}]
    \node at (0.9,0.4) {$-(5+k)$};
    \node at (1.9,0.4) {$-2$};
		\node at (3.9,0.4) {$-2$};
		\node at (4.9,0.4) {$-3$};
		\node at (5.9,0.4) {$-2$};
		\node at (7.9,0.4) {$-2$};
    \node (A1) at (1,0) {$\bullet$};
    \node (A2) at (2,0) {$\bullet$};
		\node (A3) at (3,0) {$\cdots$};
		\node (A4) at (4,0) {$\bullet$};
		\node (A5) at (5,0) {$\bullet$};
		\node (A6) at (6,0) {$\bullet$};
		\node (A7) at (7,0) {$\cdots$};
		\node (A8) at (8,0) {$\bullet$};
    \path (A1) edge [-] node [auto] {$\scriptstyle{}$} (A2);
		\path (A2) edge [-] node [auto] {$\scriptstyle{}$} (A3);
		\path (A3) edge [-] node [auto] {$\scriptstyle{}$} (A4);
		\path (A4) edge [-] node [auto] {$\scriptstyle{}$} (A5);
		\path (A5) edge [-] node [auto] {$\scriptstyle{}$} (A6);
		\path (A6) edge [-] node [auto] {$\scriptstyle{}$} (A7);
		\path (A7) edge [-] node [auto] {$\scriptstyle{}$} (A8);
		\draw [decorate,decoration={brace,amplitude=5pt,mirror,raise=2ex}] (2,0) -- (4,0) node[midway,yshift=-2em]{$s$};
		\draw [decorate,decoration={brace,amplitude=5pt,mirror,raise=2ex}] (6,0) -- (8,0) node[midway,yshift=-2em]{$k+2$};
  \end{tikzpicture} \\
D_{k,s}&:=\begin{tikzpicture}[xscale=1.0,yscale=1,baseline={(0,0)}]
    \node at (0.9,0.4) {$-(2+s)$};
    \node at (1.9,-0.4) {$-(5+k)$};
		\node at (2.9,0.4) {$-2$};
		\node at (4.9,0.4) {$-2$};
		\node at (5.9,0.4) {$-3$};
		\node at (6.9,0.4) {$-2$};
		\node at (8.9,0.4) {$-2$};
    \node (A1) at (1,0) {$\bullet$};
    \node (A2) at (2,0) {$\bullet$};
		\node (A3) at (3,0) {$\bullet$};
		\node (A4) at (4,0) {$\cdots$};
		\node (A5) at (5,0) {$\bullet$};
		\node (A6) at (6,0) {$\bullet$};
		\node (A7) at (7,0) {$\bullet$};
		\node (A8) at (8,0) {$\cdots$};
		\node (A9) at (9,0) {$\bullet$};
    \path (A1) edge [-] node [auto] {$\scriptstyle{}$} (A2);
		\path (A2) edge [-] node [auto] {$\scriptstyle{}$} (A3);
		\path (A3) edge [-] node [auto] {$\scriptstyle{}$} (A4);
		\path (A4) edge [-] node [auto] {$\scriptstyle{}$} (A5);
		\path (A5) edge [-] node [auto] {$\scriptstyle{}$} (A6);
		\path (A6) edge [-] node [auto] {$\scriptstyle{}$} (A7);
		\path (A7) edge [-] node [auto] {$\scriptstyle{}$} (A8);
		\path (A8) edge [-] node [auto] {$\scriptstyle{}$} (A9);
		\draw [decorate,decoration={brace,amplitude=5pt,mirror,raise=2ex}] (3,0) -- (5,0) node[midway,yshift=-2em]{$k$};
		\draw [decorate,decoration={brace,amplitude=5pt,mirror,raise=2ex}] (7,0) -- (9,0) node[midway,yshift=-2em]{$s$};
		\end{tikzpicture}
\end{align*}

Let
\begin{align*}
A&:=\{A_k \mid k\ge 0\}, \\
B&:=\{B_k \mid k\ge 0\}, \\
C&:=\{C_{k,s} \mid k,s\ge 0\}, \\
D&:=\{D_{k,s} \mid k,s\ge 0\}.
\end{align*}

We also let
\begin{align*}
C'&:=\{C_{k,s} \mid k\ge s\ge 0\}\subset C, \\
D'&:=\{D_{k,s} \mid k,s\ge 0, \ k+1\ge s, \ s\neq 2, \ (k,s)\neq(0,0)\}\subset D.
\end{align*}

The main theorem of this paper is:

\begin{thm}\label{thm:main}
Let $X(p,q)$ be the canonical negative-definite plumbing of $L=L(p,q)$. The following are equivalent:
\begin{enumerate}[label=(\roman*), font=\upshape]
\item\label{it:min_filling} every symplectic filling $X$ of $(L,\xi_{\mathrm{st}})$ satisfies $b_2(X) \geq b_2(X(p,q))$;
\item\label{it:blow_up} every symplectic filling $X$ of $(L,\xi_{\mathrm{st}})$ is a blowup\footnote{Here and throughout, a sequence of blowups is allowed
to be empty.} of the canonical negative-definite plumbing;
\item\label{it:combinatorial} the canonical negative-definite linear plumbing graph associated to $L$ does not contain any configurations in $A\cup B\cup C'\cup D'$ as an induced subgraph.
\end{enumerate}
\end{thm}

\begin{rem}\label{rem:thm}
The theorem is still true if we replace $A\cup B\cup C'\cup D'$ by $A\cup B\cup C\cup D$. In $C$, when $k<s$, it contains $A_{k+1}$ as an induced subgraph. In $D$, when $k+1<s$, it contains $C_{k,k}$; when $s=2$, it contains $A_0$; when $(k,s)=(0,0)$, it contains $A_1$. We write $A\cup B\cup C'\cup D'$ instead of $A\cup B\cup C\cup D$ only to make the list of configurations minimal, in the sense that none of these configurations contain another as an induced subgraph.
\end{rem}

This shows that the nature of this problem for symplectic fillings is fundamentally different from the same problem for smooth negative-definite fillings, which was solved in \cite{definite}:

\begin{cor}\label{cor:infinite}
There does not exist a finite list of configurations such that not containing any configurations from the list as an induced subgraph is equivalent to the canonical negative-definite plumbing having minimum $b_2$ across all symplectic fillings.
\end{cor}

We prove Corollary \ref{cor:infinite} at the end of the paper.

\begin{rem}\label{rem:AMP}
All configurations in \cite{definite} except for
$$\begin{tikzpicture}[xscale=1.0,yscale=1,baseline={(0,0)}]
    \node at (0.9,0.4) {$-2$};
    \node at (1.9,0.4) {$-2$};
    \node (A1) at (1,0) {$\bullet$};
    \node (A2) at (2,0) {$\bullet$};
  \end{tikzpicture}\quad\text{ and }\quad\begin{tikzpicture}[xscale=1.0,yscale=1,baseline={(0,0)}]
    \node at (0.9,0.4) {$-3$};
    \node at (1.9,0.4) {$-2$};
		\node at (2.9,0.4) {$-2$};
    \node (A1) at (1,0) {$\bullet$};
    \node (A2) at (2,0) {$\bullet$};
		\node (A3) at (3,0) {$\bullet$};
		\path (A2) edge [-] node [auto] {$\scriptstyle{}$} (A3);
  \end{tikzpicture}$$
are contained in $A\cup B\cup C'\cup D'$. The 8 connected configurations in \cite{definite} are $A_0$, $A_1$, $A_2$, $B_0$, $B_1$, $B_2$, $D_{0,1}$, and $C_{0,0}$ respectively.
\end{rem}

We use the term $(-n)$-sphere to mean an embedded sphere with self-intersection $-n$.

\subsection*{Applications}

One application of Theorem~\ref{thm:main}, combined with Remark~\ref{rem:AMP}, is that it gives a method of easily generating examples of smooth negative-definite fillings that cannot be realized symplectically. This can be achieved by simply taking any linear plumbing that does not contain any configurations in $A\cup B\cup C'\cup D'$, but contains at least one of
$$\begin{tikzpicture}[xscale=1.0,yscale=1,baseline={(0,0)}]
    \node at (0.9,0.4) {$-2$};
    \node at (1.9,0.4) {$-2$};
    \node (A1) at (1,0) {$\bullet$};
    \node (A2) at (2,0) {$\bullet$};
  \end{tikzpicture}\quad\text{ or }\quad\begin{tikzpicture}[xscale=1.0,yscale=1,baseline={(0,0)}]
    \node at (0.9,0.4) {$-3$};
    \node at (1.9,0.4) {$-2$};
		\node at (2.9,0.4) {$-2$};
    \node (A1) at (1,0) {$\bullet$};
    \node (A2) at (2,0) {$\bullet$};
		\node (A3) at (3,0) {$\bullet$};
		\path (A2) edge [-] node [auto] {$\scriptstyle{}$} (A3);
  \end{tikzpicture},$$
	and then perform rational blow down according to \cite[\S 2]{definite}. A simple example would be:
	
\begin{cor}
For every lens space in the family $L=L(4k,2k+1)$ with $k\ge 5$, it bounds a smooth negative-definite manifold $X$ with $b_2(X)=2$ that is not diffeomorphic to any symplectic filling of $(L,\xi_{\mathrm{st}})$.
\end{cor}

\begin{proof}
The canonical negative-definite plumbing graph of $L(4k,2k+1)$ is
$$\begin{tikzpicture}[xscale=1.0,yscale=1,baseline={(0,0)}]
    \node at (0.9,0.4) {$-2$};
    \node at (1.9,0.4) {$-(k+1)$};
		\node at (2.9,0.4) {$-2$};
    \node (A1) at (1,0) {$\bullet$};
    \node (A2) at (2,0) {$\bullet$};
		\node (A3) at (3,0) {$\bullet$};
		\path (A1) edge [-] node [auto] {$\scriptstyle{}$} (A2);
		\path (A2) edge [-] node [auto] {$\scriptstyle{}$} (A3);
  \end{tikzpicture}$$
By tubing the two $(-2)$-spheres together, we obtain a smoothly embedded $(-4)$-sphere. Performing a smooth rational blowdown along this sphere produces a smooth negative-definite $X$ with $b_2(X)=2$ (see section 2 of \cite{definite}). However, the plumbing graph does not contain any configurations in $A\cup B\cup C'\cup D'$ as an induced subgraph, and therefore by Theorem~\ref{thm:main}, the minimum $b_2$ among all symplectic fillings of $(L,\xi_{\mathrm{st}})$ is 3.
\end{proof}

Since Bhupal and Ozbagci \cite{bhupal2016symplectic} proved that every minimal\footnote{Here, ``minimal'' means that the filling contains no symplectically embedded $(-1)$-sphere. This differs from minimality of $b_2$ among all fillings, which is the main topic of this paper. However, a filling with minimal $b_2$ is minimal in this sense, since otherwise symplectically blowing down a $(-1)$-sphere would produce a filling with smaller $b_2$.} symplectic filling of the standard contact structure on lens spaces can be obtained from the canonical plumbing by a sequence of symplectic rational blowdowns, Theorem~\ref{thm:main} also yields the following criterion for when the canonical plumbing admits a symplectic rational blowdown:

\begin{cor}\label{cor:blowdown}
Let $X(p,q)$ be the canonical negative-definite plumbing of $L(p,q)$. Then $X(p,q)$ admits a symplectic rational blowdown if and only if its plumbing graph contains a configuration in $A\cup B\cup C'\cup D'$ as an induced subgraph.
\end{cor}

\begin{proof}
If $X(p,q)$ admits a symplectic rational blowdown, then the result is a symplectic filling of $(L(p,q),\xi_{\mathrm{st}})$ with a smaller $b_2$. Hence, by Theorem~\ref{thm:main}, the plumbing graph contains a configuration in $A\cup B\cup C'\cup D'$ as an induced subgraph.

Conversely, if the plumbing graph contains a configuration in $A\cup B\cup C'\cup D'$ as an induced subgraph, then by Theorem~\ref{thm:main} there is a symplectic filling of $(L(p,q),\xi_{\mathrm{st}})$ with a smaller $b_2$. Choose such a filling with minimal $b_2$. It is symplectically minimal, since otherwise symplectically blowing down a $(-1)$-sphere would yield a filling with smaller $b_2$. By \cite{bhupal2016symplectic}, it is obtained from the canonical plumbing by a sequence of symplectic rational blowdowns. Since its $b_2$ is smaller, this sequence is nonempty. Hence, the canonical plumbing admits a symplectic rational blowdown.
\end{proof}

\begin{rem}
Corollary~\ref{cor:blowdown} is an existence statement. The rational blowdown configuration need not be apparent from the canonical plumbing graph itself. See Remark~\ref{rem:final} for an explicit construction.
\end{rem}
	
While this paper concerns lens spaces for which the canonical negative-definite plumbing $X(p,q)$ has minimal $b_2$ among symplectic fillings of $(L(p,q),\xi_{\mathrm{st}})$, one may also ask about lens spaces where the canonical negative-definite plumbing is \textit{nearly} $b_2$-minimal. Precisely, one may ask:

\begin{ques}
Fix some small $k\geq 1$. Which lens spaces satisfy the property that for every symplectic filling $X$ of $(L(p,q),\xi_{\mathrm{st}})$, we have $b_2(X) \ge b_2(X(p,q))-k$? Can we perform a similar classification based on forbidden induced subgraphs? 
\end{ques}

The author believes that such a generalization can be done with techniques demonstrated in this paper, with more cumbersome case splitting needed in the combinatorial analysis steps, classifying tuples with at least $k+1$ middle blowups instead of at least one middle blowup (see the proof outline subsection below).

Another potential application is that Section \ref{sec:setup} provides a systematic way to organize sequences of blowups. By showing that any strict blowup sequence can be reordered in a specific way, we obtain a normal form for blowup processes. This normal form can be used more broadly to analyze and obstruct constructions that rely on iterated blowups, for example the tuples corresponding to symplectic fillings studied in this paper, and potentially more generally in settings where iterated blowups play a role, such as in low-dimensional topology and algebraic geometry.

\subsection*{Proof outline for Theorem~\ref{thm:main}}

Lisca \cite{Lisca:2008-1} classified symplectic fillings of the standard contact structures of lens spaces according to a combinatorial condition called \textit{admissible $k$-tuples}. The equivalence \ref{it:min_filling}$\Leftrightarrow$\ref{it:blow_up} follows directly from Lisca's work. The majority of the work is devoted to proving \ref{it:combinatorial}$\Rightarrow$\ref{it:min_filling}. To do so, we show that the negation of \ref{it:min_filling} implies the negation of \ref{it:combinatorial}. First, we show that a symplectic filling $X$ with $b_2(X)<b_2(X(p,q))$ corresponds to a tuple with at least one \textit{middle blowup} (see Definition~\ref{def:middle blowup}). Then, we analyze the subtuples of such a tuple (see Lemma~\ref{lem:n covering}). Then, we convert it into a statement about the plumbing graph. We provide two equally quick ways to prove \ref{it:min_filling}$\Rightarrow$\ref{it:combinatorial}. One way reverses the argument above: given an induced subgraph of the plumbing graph that is in $A\cup B\cup C'\cup D'$, we construct a $k$-tuple that corresponds to a symplectic filling $X$ with $b_2(X)<b_2(X(p,q))$. The other approach is to explicitly exhibit a symplectic rational blowdown associated to each configuration in $A\cup B\cup C'\cup D'$.

\subsection*{Structure of the article}

In Section~\ref{sec:lisca}, we recall some details from Lisca's work on symplectic fillings of lens spaces \cite{Lisca:2008-1}. In Section~\ref{sec:setup}, we investigate ways to rearrange the order of doing blowups without changing the result. In Section~\ref{sec:dual}, we recall the relationship between the canonical negative-definite plumbing of $L(p,q)$ and the canonical negative-definite plumbing of $L(p,p-q)$. In Section~\ref{sec:main}, we combine the above ideas and perform the technical combinatorial analysis needed to prove Theorem~\ref{thm:main}.

\subsection*{Notations}

Throughout this paper, $a^{[b]}$ denotes $b$ consecutive copies of $a$. For example, $(1,2^{[3]},4)$ denotes the tuple $(1,2,2,2,4)$, and $(1,2^{[0]},4)$ denotes the tuple $(1,4)$.

\subsection*{Acknowledgements} 

The author would like to thank JungHwan Park, Marco Golla, and Brendan Owens for helpful comments on an earlier draft of this paper. The majority of this work was carried out while the author was a postdoctoral researcher at KAIST. The author is partially supported by the Samsung Science and Technology Foundation (SSTF-BA2102-02) and by the NRF grant RS-2025-00542968.

\section{Lisca's construction}\label{sec:lisca}

We recall some key theorems and lemmas from \cite{Lisca:2008-1} that this paper will be built on.

We first define the concept of performing a \textit{strict blowup} on a tuple, which means performing a blowup not at the left end of the tuple. Formally:

\begin{defn}[{\cite[Def.~2.1]{Lisca:2008-1}}]
Let $\textbf{n}=(n_1,\dots,n_k)$ be a tuple of non-negative integers for some $k\ge 1$. We say that $\textbf{n}'$ is obtained from $\textbf{n}$ by a \textit{strict blowup} if $\textbf{n}'$ is in the form of either
$$(n_1,\dots,n_{s-1},n_s+1,1,n_{s+1}+1,n_{s+2},\dots,n_k)$$
for some $1\le s\le k-1$, or
$$(n_1,\dots,n_{k-1},n_k+1,1).$$
\end{defn}

Let $p>q\ge 1$ be coprime integers. Consider the Hirzebruch-Jung expansion of $p\slash(p-q)$:
$$\dfrac{p}{p-q}= b_1 - \cfrac{1}{b_2 - \cfrac{1}{\ddots - \cfrac{1}{b_k}}}$$
where $b_1,\dots,b_k\ge 2$. Throughout this paper, we denote this expansion as $[b_1,\dots,b_k]^-$. Note that the fraction on the left hand side is $p\slash(p-q)$, not $p\slash q$. Whenever the notation $b_2$ arises, it should be clear from the context whether it refers to the the entry in the Hirzebruch-Jung expansion or refers to the second Betti number.

Let $Z_{p,q}\subset\Z^k$ be the set of $k$-tuples $(n_1,\dots,n_k)$ satisfying both of the following:
\begin{enumerate}
\item $0\le n_i\le b_i$ for all $i$; and
\item $(n_1,\dots,n_k)$ is obtained from $(0)$ by a sequence of strict blowups.
\end{enumerate}

At first glance, the definition of $Z_{p,q}$ here seems to be different from \cite{Lisca:2008-1}. However, they are equivalent due to \cite[Lem. 2.2]{Lisca:2008-1}.

In \cite{Lisca:2008-1}, for each $\textbf{n}\in Z_{p,q}$, Lisca constructed a 4-manifold $W_{p,q}(\textbf{n})$ with $L(p,q)$ as boundary. In this paper, whenever we say that a manifold with boundary $L(p,q)$ ``corresponds to'' a tuple $\textbf{n}$, or a tuple $\textbf{n}$ ``corresponds to'' a manifold, we mean that the manifold is orientation-preserving diffeomorphic to $W_{p,q}(\textbf{n})$.

The manifold $W_{p,q}(\textbf{n})$ carries a symplectic structure that fills $(L(p,q),\xi_{\mathrm{st}})$ \cite[Thm. 1.1(b)]{Lisca:2008-1}, and satisfies
$$b_2(W_{p,q}(\textbf{n}))=-1+\Sigma_i(b_i-n_i).$$
When $k\ge 2$, the canonical negative-definite plumbing $X(p,q)$ corresponds to the $k$-tuple $(1,2,2,\dots,2,1)$, where the number of $2$'s is $k-2$. When $k=1$, the corresponding $k$-tuple is $(0)$. In both cases, the sum of the entries is $2k-2$.

Hence, we have
$$b_2(X(p,q))=-1+\Sigma_i b_i-(2k-2).$$
Therefore,
$$b_2(W_{p,q}(\textbf{n}))=b_2(X(p,q))+(2k-2)-\Sigma_i n_i.$$
The following result by Lisca will serve as a foundation for this paper.

\begin{lem}\cite[Thm. 1.1 (a)]{Lisca:2008-1}\label{lem:lisca}
Let $(W,\omega)$ be a symplectic filling of $(L(p,q),\xi_{\mathrm{st}})$. Then, for some $\textbf{n}\in Z_{p,q}$, $W$ is orientation-preserving diffeomorphic to a manifold obtained from $W_{p,q}(\textbf{n})$ by a (possibly empty) sequence of smooth blowups.
\end{lem}

A corollary of Lemma~\ref{lem:lisca} is that the statement Theorem~\ref{thm:main}\ref{it:min_filling} is equivalent to saying that there is no $\textbf{n}\in\Z_{p,q}$ satisfying $b_2(W_{p,q}(\textbf{n}))<b_2(X(p,q))$.

\section{Set up}\label{sec:setup}

Starting from here, until the end of this paper, $\textbf{n}$ always represent a tuple obtained from $(0)$ through a sequence of strict blowups.

The goal of this section is to break down the blowup process into components that allow us to write down the blowup process in a nice way. Consider the sequence of strict blowups:

$$(0)\rightarrow(1,1)\rightarrow(2,1,2)\rightarrow(2,1,3,1)\rightarrow(2,1,4,1,2)\rightarrow(2,1,5,1,2,2)$$
$$\rightarrow(3,1,2,5,1,2,2)\rightarrow(3,1,2,5,2,1,3,2)\rightarrow(3,2,1,3,5,2,1,3,2).$$

One obvious way to describe the blowup process would be:
\begin{enumerate}
\item blowup at position 1;
\item blowup at position 1;
\item blowup at position 3;
\item blowup at position 3;
\item blowup at position 3;
\item blowup at position 1;
\item blowup at position 5;
\item blowup at position 2.
\end{enumerate}

However, this description does not provide much insight into the blowup process. After making various definitions in this chapter, we will be able to rewrite this as:
\begin{enumerate}
\item end blowup;
\item shallow blowup at position 1, creating trough $T_1$;
\item end blowup;
\item shallow blowup at position 3, creating trough $T_2$;
\item deep blowup at left of trough $T_2$;
\item deep blowup at left of trough $T_1$;
\item deep blowup at right of trough $T_2$;
\item deep blowup at right of trough $T_1$,
\end{enumerate}

Through lemmas in this section that describe how actions can be swapped, we will then be able to arrange this into:
\begin{enumerate}
\item end blowup 2 times;
\item shallow blowup at position 1, creating trough $T_1$;
\item deep blowup at left of trough $T_1$;
\item deep blowup at right of trough $T_1$,
\item shallow blowup at position 5, creating trough $T_2$;
\item deep blowup at left of trough $T_2$;
\item deep blowup at right of trough $T_2$;
\end{enumerate}

Explicitly, the reordered sequence is:

$$(0)\rightarrow(1,1)\rightarrow(1,2,1)\rightarrow(2,1,3,1)\rightarrow(3,1,2,3,1)\rightarrow(3,2,1,3,3,1)$$
$$\rightarrow(3,2,1,3,4,1,2)\rightarrow(3,2,1,3,5,1,2,2)\rightarrow(3,2,1,3,5,2,1,3,2).$$

While there may be ways to describe shallow blowups (Definition~\ref{def:trough}) not based on positions so that the description does not change under swapping, developing such mechanism is not needed for the proof of Theorem~\ref{thm:main} and therefore we will not do that.

\subsection{Different types of blowups}

We start with the following definition.

\begin{defn}\label{def:middle blowup}
If a strict blowup occurs at the right end, i.e. the $(n_1,\dots,n_{k-1},n_k+1,1)$ case in the definition of strict blowup in the previous section, we say that the strict blowup is an \textit{end blowup}. Otherwise, it is a \textit{middle blowup}.
\end{defn}

Before we go further, we first note the significance of dividing strict blowups into end blowups and middle blowups.

\begin{lem}\label{lem:thm equiv 1}
In Theorem~\ref{thm:main}, statement~\ref{it:min_filling} is equivalent to no element in $\textbf{n}\in Z_{p,q}$ is obtained from a strict blowup process from $(0)$ that involves at least one middle blowup.
\end{lem}

\begin{proof}
This proof is similar to the argument near the end of the proof of \cite[Cor. 1.5]{forbidden} (the proof is in Section 7).

Recall from the previous section that
$$b_2(W_{p,q}(\textbf{n}))=b_2(X(p,q))+(2k-2)-\Sigma_i n_i.$$

Observe that a $k$-tuple in $Z_{p,q}$ is obtained by doing strict blowup from $(0)$ for $k-1$ times. Every middle blowup increases $\Sigma_i n_i$ by 3, while every end blowup increases $\Sigma_i n_i$ by 2. Hence, we have
$$\Sigma_i n_i=(2k-2)+(\text{number of middle blowups performed}).$$

Therefore, the number of middle blowups performed equals to $b_2(X(p,q))-b_2(W_{p,q}(\textbf{n}))$. The result follows from Lemma~\ref{lem:lisca}.
\end{proof}

\begin{cor}\label{cor:1 2 equiv}
Let $L(p,q)$ be a lens space. Then, in Theorem~\ref{thm:main}, \ref{it:min_filling} is equivalent to \ref{it:blow_up}.
\end{cor}

\begin{proof}
In Theorem~\ref{thm:main}, condition \ref{it:blow_up} implies \ref{it:min_filling}, since blowing up a manifold does not reduce $b_2$. If \ref{it:min_filling} holds, then Lemma~\ref{lem:thm equiv 1} implies that any $\textbf{n}\in Z_{p,q}$ must come from $(0)$ with only end blowups. This gives $\textbf{n}=(1,2,2,\dots,2,1)$, which corresponds to the canonical negative-definite plumbing. Statement \ref{it:blow_up} follows from Lemma~\ref{lem:lisca}.
\end{proof}

Now, we prove the following lemma which allows us to isolate the effect of middle blowups and end blowups.

\begin{lem}\label{lem:end blowups first}
If $\textbf{n}$ is obtained from $(0)$ via a sequence of strict blowups consisting of $x$ end blowups and $y$ middle blowups, then $\textbf{n}$ can be obtained from $(0)$ by first performing $x$ end blowups, followed by $y$ middle blowups.
\end{lem}

In other words, we can do all the end blowups first.

\begin{proof}
We show that if $\textbf{n}'$ is obtained from $\textbf{n}$ via a middle blowup followed by an end blowup, then it can be obtained from $\textbf{n}$ via an end blowup followed by a middle blowup. This proves the lemma because then we can keep swapping the order of end and middle blowups until all the end blowups are performed before the middle blowups.

Let $\textbf{n}=(n_1,\dots,n_k)$. Performing a middle blowup followed by an end blowup, we get
$$(n_1,\dots,n_{s-1},n_s+1,1,n_{s+1}+1,n_{s+2},\dots,n_{k-1},n_k+1,1)$$
or
$$(n_1,\dots,n_{k-2},n_{k-1}+1,1,n_k+2,1).$$
In both cases, the same can be obtained by first performing an end blowup followed by a middle blowup.
\end{proof}

We further subdivide middle blowups in the following definition.

\begin{defn}\label{def:trough}
A \textit{trough} in $\textbf{n}$ is an entry with value $1$ that is not located at the left or right end. A \textit{deep} blowup is a middle blowup performed next to a trough. A \textit{shallow} blowup is a middle blowup performed not next to a trough.
\end{defn}

Now we start investigating properties of $\textbf{n}$.

\begin{lem}
If $\textbf{n}$ is not $(0)$, then all entries of $\textbf{n}$ are positive integers. 
\end{lem}

\begin{proof}
From $(0)$, after one blowup, it becomes $(1,1)$. Every blowup introduces an entry with value $1$ and increases some other values, so it can never produce a $0$ or a negative entry.
\end{proof}

\begin{lem}\label{lem:no adjacent troughs}
Two troughs cannot be adjacent to each other.
\end{lem}

\begin{proof}
We show that a strict blowup will never produce a pair of adjacent troughs. Since $\textbf{n}$ is obtained from $(0)$ through a sequence of strict blowups, it follows that $\textbf{n}$ has no adjacent troughs.

In a middle blowup, the value of $n_s+1$ and $n_{s+1}+1$ in
$$(n_1,\dots,n_{s-1},n_s+1,1,n_{s+1}+1,n_{s+2},\dots,n_k)$$
is at least 2. In an end blowup, the value of $n_k+1$ in 
$$(n_1,\dots,n_{k-1},n_k+1,1)$$
is at least 2, unless the blowup is $(0)\rightarrow(1,1)$.

In all cases, no new pair of adjacent troughs is produced.
\end{proof}

Now, we know how each type of strict blowup affects the number of troughs. End blowups preserve the number of troughs. A shallow blowup increases the number of troughs by 1. And because of Lemma~\ref{lem:no adjacent troughs}, we know that deep blowups preserve the number of troughs.

\subsection{A tuple growing over time} Consider the sequence of tuples $(0)=\textbf{n}^1,\textbf{n}^2,\dots,\textbf{n}^k=\textbf{n}$, where each $\textbf{n}^{t+1}$ is a strict blowup of $\textbf{n}^t$. We think of $\textbf{n}^t$ as a tuple that is growing over time $t$. In each blowup, we insert a new entry into the tuple (the entry with value $1$ in the definition of blowup). An entry adjacent to it is ``the same entry'', with its value increased by $1$ when going from time $t$ to time $t+1$.

For example, consider the blowup $\textbf{n}^3=(1,2,1)\rightarrow\textbf{n}^4=(1,3,1,2)$. The leftmost 1 in $\textbf{n}^3$ and $\textbf{n}^4$ is the same entry. The 2 in $\textbf{n}^3$ and the 3 in $\textbf{n}^4$ is the same entry, with its value changed during the blowup. The second 1 in $\textbf{n}^4$ is a newly inserted entry that did not exist in time $t=3$. The rightmost 1 in $\textbf{n}^3$ is the same entry as the 2 in $\textbf{n}^4$, with its value changed during the blowup.

For each $t$, we label its troughs with labels $T_1,T_2,\dots$. The labels are defined with the following description. At time $0$, there is no trough, so there is no labels. Suppose we have labeled all the troughs for $\textbf{n}^t$, we label the troughs for $\textbf{n}^{t+1}$ according to the following rules:
\begin{enumerate}
\item If $\textbf{n}^{t+1}$ is obtained from $\textbf{n}^t$ through an end blowup, no new trough is introduced. Every trough in $\textbf{n}^{t+1}$ carries the same label as they had in $\textbf{n}^t$.
\item If $\textbf{n}^{t+1}$ is obtained from $\textbf{n}^t$ through a shallow blowup: Let $k\ge 0$ be the number of troughs that $\textbf{n}^t$ has. The newly inserted entry gets the label $T_{k+1}$. All troughs in $\textbf{n}^t$ are still troughs in $\textbf{n}^{t+1}$, and they carry the same label as they had in $\textbf{n}^t$.
\item If $\textbf{n}^{t+1}$ is obtained from $\textbf{n}^t$ through a deep blowup: By Lemma~\ref{lem:no adjacent troughs}, we know that the deep blowup was performed next to a unique trough $T_i$ in $\textbf{n}^t$ for some $i$. After the blowup, the entry that carried the label $T_i$ at time $t$ is no longer a trough, while the newly inserted entry at time $t+1$ is a trough. This newly inserted entry takes the label $T_i$ at time $t+1$, while the entry that had the label $T_i$ at time $t$ no longer has a label at time $t+1$ because it is no longer a trough. All other troughs carry the same label as they had in $\textbf{n}^t$.
\end{enumerate}

We think of entries that carry the same label across different time as ``the same trough''. A shallow blowup creates a new trough. Once a trough is created, it stays forever. A deep blowup next to a trough simply moves the trough to the new entry instead of destroying the trough.

After applying these labels, now every deep blowup can be written as ``deep blowup at left of trough $T_i$'' or ``deep blowup at right of trough $T_i$'' for some $i$, with left or right depending on whether the blowup is performed on the left side gap or the right side gap of the trough.

\begin{lem}\label{lem:deep swap}
Let $i\neq j$. If $\textbf{n}^{t+2}$ is obtained from $\textbf{n}^t$ by first performing a deep blowup next to $T_i$ and then a deep blowup next to $T_j$, then $\textbf{n}^{t+2}$ can also be obtained from $\textbf{n}^t$ by first performing a deep blowup next to $T_j$ and then a deep blowup next to $T_i$.
\end{lem}

\begin{proof}
Without loss of generality, we can assume that $T_i$ is on the left of $T_j$. When $T_i$ and $T_j$ are not close to each other, the two blowups have no influence on each other, so there is nothing to check. The only case needed to be checked is when $\textbf{n}^t$ is in the form $(\dots,T_i,x,T_j,\dots)$ for some $x\ge 2$, and we do deep blowups on the right side of $T_i$ and the left side of $T_j$. Doing deep blowup on the right side of $T_i$ first followed by doing deep blowup on the left side of $T_j$ produces $(\dots,2,T_i,x+2,T_j,2,\dots)$, but so is doing deep blowup on the left side of $T_j$ first followed by doing deep blowup on the right side of $T_i$. Therefore, they produce the same result.
\end{proof}

\begin{lem}\label{lem:shallow swap}
Let $i<j$. If $\textbf{n}^{t+2}$ is obtained from $\textbf{n}^t$ by first performing a shallow blowup to create trough $T_j$ and then performing a deep blowup next to $T_i$, then $\textbf{n}^{t+2}$ can also be obtained from $\textbf{n}^t$ by first performing a deep blowup next to $T_i$ and then performing a shallow blowup.
\end{lem}

\begin{proof}
Same as the above proof, there is nothing to check if the blowups are performed in different parts of the tuple. Without loss of generality, we can assume that $T_i$ is on the left of $T_j$. The only interesting cases are:
$$(\dots,T_i,x,y,\dots)\rightarrow(\dots,T_i,x+1,T_j,y+1,\dots)\rightarrow(\dots,2,T_i,x+2,T_j,y+1,\dots)$$
which can be replaced with
$$(\dots,T_i,x,y,\dots)\rightarrow(\dots,2,T_i,x+1,y,\dots)\rightarrow(\dots,2,T_i,x+2,T_j,y+1,\dots).$$
\end{proof}

\begin{lem}\label{lem:standard order}
If $\textbf{n}$ is obtained from $(0)$ through a sequence of strict blowups, then $\textbf{n}$ can be obtained from $(0)$ through a sequence of strict blowups in the following order:
\begin{enumerate}
\item a sequence of end blowups;
\item a shallow blowup to create trough $T_1$;
\item a sequence of deep blowups next to $T_1$;
\item a shallow blowup to create trough $T_2$;
\item a sequence of deep blowups next to $T_2$;
\item a shallow blowup to create trough $T_3$;
\item a sequence of deep blowups next to $T_3$;
\item[] etc$\dots$.
\end{enumerate}
Equivalently, it is saying that we can do all the end blowups first, then every time after we do a shallow blowup, we can perform all the deep blowups next to that newly created trough first before doing anything else.
\end{lem}

\begin{proof}
Such a blowup order can be achieved by using Lemmas~\ref{lem:end blowups first}, \ref{lem:deep swap}, and \ref{lem:shallow swap} to swap the order of blowups.
\end{proof}

\section{Dual embedding and adjusted values}\label{sec:dual}

Since $Z_{p,q}$ is defined using $p\slash(p-q)$, while Theorem~\ref{thm:main} statement~\ref{it:combinatorial} is about $p\slash q$, we shall connect the two.

We borrow the concept of \textit{adjusted weight} from \cite{definite}. In \cite{definite}, in a graph where every vertex $v$ has a weight $w(v)$, the \textit{adjusted weight} of $v$ is $w(v)-d(v)$, where $d(v)$ is the degree of $v$.

Similarly, we define the concept of the \textit{adjusted value} of the entries of a $k$-tuple:

\begin{defn}
Let $\textbf{n}=(n_1,\dots,n_k)$ be a $k$-tuple. If $k\ge 2$, we define
$$adjusted(\textbf{n}):=(n_1-1,n_2-2,n_3-2,\dots,n_{k-1}-2,n_k-1).$$
If $k=1$, then $adjusted(\textbf{n}):=\textbf{n}$.

We call the entries in $adjusted(\textbf{n})$ \textit{adjusted values} of the corresponding entries in $\textbf{n}$.
\end{defn}

Under this language, a trough is an entry whose adjusted value is $-1$.

\begin{lem}[Equivalent to Lemma~3.5 in \cite{definite}]\label{lem:dual}
Let $p>q>0$ be coprime integers. Let $[c_1,\dots,c_{l_1}]^-$ be the Hirzebruch-Jung expansion of $p\slash q$, and $[d_1,\dots,d_{l_2}]^-$ be the Hirzebruch-Jung expansion of $p\slash(p-q)$.

Let $a_i\ge 0$, $b_i\ge 3$ be integers satisfying
$$(c_1,\dots,c_{l_1})=(2^{[a_0]},b_1,2^{[a_1]},\dots,b_k,2^{[a_k]}).$$
Then,
$$adjusted((d_1,\dots,d_{l_2}))=(a_0+1,0^{[b_1-3]},a_1+1,\dots,0^{[b_k-3]},a_k+1).$$
\end{lem}

\begin{proof}
The adjusted values of $d_i$'s here correspond to the adjusted weights $d_i$ in \cite[Lemma 3.5]{definite}.
\end{proof}

\begin{rem}
The two expressions in Lemma~\ref{lem:dual} determine one another: the tuple $(c_1,\dots,c_{l_1})$ uniquely determines the coefficients $a_i$ and $b_i$, and the tuple $adjusted((d_1,\dots,d_{l_2}))$ uniquely determines them as well. Thus, Lemma~\ref{lem:dual} can be used to pass between these two tuples in either direction.
\end{rem}

\section{Proof of Theorem~\ref{thm:main}}\label{sec:main}

In this section, we prove Theorem~\ref{thm:main}. We first perform combinatorial analysis up to Lemma~\ref{lem:LR classification}, independent from the content in the previous two sections. Then, we combine it with the content of Section~\ref{sec:setup} to prove Lemma~\ref{lem:n covering}. In Lemmas~\ref{lem:case 1}, \ref{lem:case 2}, and \ref{lem:case 3}, we use the content of Section~\ref{sec:dual} to translate the results back to the statement of Theorem~\ref{thm:main}.

Recall that in the definition of $Z_{p,q}$ in Section~\ref{sec:lisca}, the entries $n_i$ and $b_i$ satisfy $0\le n_i\le b_i$. To mirror that combinatorial constraint, we make the definition below. We allow the reversed block $y_{l_2+1-i}$ because the definition of induced subgraph does not depend on the direction in which the plumbing graph is visually drawn. The use of $max(x_{j+i},0)$ rather than simply $x_{j+i}$ accounts for the fact that adjusted tuples arising from blowup sequences may have negative entries, namely the $-1$ entries corresponding to troughs. On the other hand, every $b_i$ in the Hirzebruch-Jung expansion is at least $2$, so $adjusted((b_1,\dots,b_l))$ has non-negative entries. Consequently, when $(x_1,\dots,x_l)=adjusted((b_1,\dots,b_l))$, the condition $max(x_{j+i},0)\ge y_i$ implies $x_{j+i}\ge y_i$.

\begin{defn}
Let $s_1=(x_1,\dots,x_{l_1})$ be an $l_1$-tuple of integers and $s_2=(y_1,\dots,y_{l_2})$ be an $l_2$-tuple of integers. We say that $s_1$ \textit{covers} $s_2$ if $l_1\ge l_2$ and there is a consecutive block of entries $x_{j+1},\dots,x_{j+l_2}$ in $s_1$ such that either $max(x_{j+i},0)\ge y_i$ for all $1\le i\le l_2$ or $max(x_{j+i},0)\ge y_{l_2+1-i}$ for all $1\le i\le l_2$.
\end{defn}

We prove that covering is transitive:

\begin{lem}\label{lem:cover transitive}
If $s_1$ covers $s_2$ and $s_2$ covers $s_3$, then $s_1$ covers $s_3$.
\end{lem}

\begin{proof}
Reverse $s_2$ or $s_3$ if needed so that we use the ``$max(x_{j+i},0)\ge y_i$'' part of the definition in both $s_1$'s covering of $s_2$ and $s_2$'s covering of $s_3$.

Let $s_1=(x_1,\dots,x_{l_1})$, $s_2=(y_1,\dots,y_{l_2})$, and $s_3=(z_1,\dots,z_{l_3})$. First of all, $l_1\ge l_2$ and $l_2 \ge l_3$ implies $l_1\ge l_3$.

Let $j_1$ be such that $max(x_{j_1+i},0)\ge y_i$ for all $i$, and $j_2$ be such that $max(y_{j_2+i},0)\ge z_i$ for all $i$.

Then for all $i$, we have $max(x_{j_1+j_2+i},0)\ge y_{j_2+i}$, and also $max(x_{j_1+j_2+i},0)\ge 0$ by the definition of maximum. Therefore, we have 
$$max(x_{j_1+j_2+i},0)\ge max(y_{j_2+i},0)\ge z_i.$$
\end{proof}

\begin{defn}
A \textit{left-right sequence} is a (possibly empty) finite string of ``L''s and ``R''s. A left-right sequence associated with a sequence of deep blowups is the string of ``L''s and ``R''s recording whether the blowup happened at the left side or right side of the trough for each blowup.

We read a left-right sequence from left to right in chronological order. Thus $RRL$ means: first perform a deep blowup on the right of the trough, then another on the right, and finally one on the left.

Given a left-right sequence $S$, we define the tuple $apply(S)$ to be the adjusted value tuple obtained by applying the corresponding sequence deep blowups on the trough in the adjusted value tuple $(1,-1,1)$.
\end{defn}

For example, $apply()=(1,-1,1)$, $apply(R)=(1,0,-1,2)$, $apply(RRL)=(1,0,1,-1,0,3)$.

Note that $(1,-1,1)$ is the adjusted value tuple of $(2,1,2)$, $(1,0,-1,2)$ is the adjusted value tuple of $(2,2,1,3)$, and $(1,0,1,-1,0,3)$ is the adjusted value tuple of $(2,2,3,1,2,4)$. When performing blowup calculations, one could first convert an adjusted value tuple back to its original tuple, perform the blowup there, and then pass back to adjusted values. However, this is unnecessarily cumbersome. It is more convenient to work directly with adjusted value tuples: a middle blowup inserts an entry $-1$ and increases the adjacent entries by $1$, while an end blowup appends an entry $0$ at the right end and increases the adjacent entry by $1$.

\begin{lem}\label{lem:LR classification}
Let $S$ be a left-right sequence. Then, $apply(S)$ covers at least one of the following:
\begin{enumerate}
\item $(1,0^{[k]},k)$ for some $k\ge 1$;
\item $(1,k,0^{[k+1]},2)$ for some $k\ge 0$;
\item $(1,0^{[k_1]},1,0^{[k_2+1]},k_2,k_1+2)$ for some $k_1,k_2\ge 0$.
\end{enumerate}
\end{lem}

\begin{proof}
We assume that for all $k\ge 1$, $apply(S)$ does not cover any of those tuples. We try to arrive at a contradiction.

If $S$ is empty, then $apply(S)$ is $(1,-1,1)$, which covers $(1,0,1)$. So, we know that $S$ is non-empty. By symmetry, we can assume that $S$ starts with $R$.

If $S$ does not have any $L$, then $apply(S)$ is $(1,0^{[a]},-1,a+1)$, where $a$ is the length of $S$. This covers $(1,0^{[k]},k)$ with $k=a+1$. So, $S$ has at least one $L$.

Let $x$ be the number of $R$'s before the first $L$.

If there are no more $R$'s after the first $L$, let $y$ be the number of $L$'s. Then, $apply(S)$ is $(1,0^{[x-1]},y,-1,0^{[y]},x+1)$. If $y\ge x-1$, then the $(1,0^{[x-1]},y)$ part covers $(1,0^{[k]},k)$ for $k=x-1$. If $y\le x$, then the $(y,-1,0^{[y]},x+1)$ part covers $(1,0^{[k]},k)$ for $k=y+1$.

Hence, there is some $R$ after the first $L$. So, $S$ starts with $R^{[x]}L^{[y]}R$ for some $x,y\ge 1$, where $R^{[x]}$ denotes $x$ copies of $R$ and $L^{[y]}$ denotes $y$ copies of $L$.

Note that $apply(R^{[x]}L^{[y]}R)=(1,0^{[x-1]},y,0,-1,1,0^{[y-1]},x+1)$.

Since the $-1$ is after the $(y,0)$ part, any further deep blowups cannot affect the $(1,0^{[x-1]},y)$ part. Hence, if $y\ge x-1$ and $x\ge 2$, $apply(S)$ will always cover $(1,0^{[k]},k)$ for $k=x-1$.

Similarly, any further deep blowups cannot affect the $(1,0^{[y-1]},x+1)$ part except for potentially increasing the value of the entry that is currently written as ``1''. Therefore, if $x+2\ge y\ge 2$, $apply(S)$ will always cover $(1,0^{[k]},k)$ for $k=y-1$.

Therefore, we must have $x=1$ or $y=1$.

\medskip\noindent\textbf{Case 1: }{$x=y=1$.}

For any $z\ge 1$, we have $apply(RLR^{[z]})=(1,1,0^{[z]},-1,z,2)$, which covers $(1,0^{[k_1]},1,0^{[k_2+1]},k_2,k_1+2)$ for $(k_1,k_2)=(0,z)$. Therefore, there exists some $L$ that appears after the first $L$. Let $z$ be the number of $R$ between the first and second $L$.

If there are no more $R$'s after the second $L$, then $S$ is in the form $RLR^{[z]}L^{[w]}$ for some $w\ge 1$. In that case, we have $apply(S)=(1,1,0^{[z-1]},w,-1,0^{[w]},z,2)$. If $w\ge z$, then the $(1,1,0^{[z-1]},w)$ part covers $(1,0^{[k]},k)$ for $k=z$. If $z-1\ge w$, then the $(w,-1,0^{[w]},z)$ part covers $(1,0^{[k]},k)$ for $k=w+1$. So, there exists some $R$ after the second $L$, and $S$ is in the form $RLR^{[z]}L^{[w]}R\dots$ for some $w\ge 1$.

Note that $apply(RLR^{[z]}L^{[w]}R)=(1,1,0^{[z-1]},w,0,-1,1,0^{[w-1]},z,2)$.

Since the $-1$ is after the $(w,0)$ part, any further deep blowups cannot affect the $(1,1,0^{[z-1]},w)$ part. Hence, if $w\ge z$, $apply(S)$ will always cover $(1,0^{[k]},k)$ for $k=w$.

Similarly, any further deep blowups cannot affect the $(1,0^{[w-1]},z)$ part except for potentially increasing the value of the entry that is currently written as ``1''. Therefore, if $z+1\ge w\ge 2$, $apply(S)$ will always cover $(1,0^{[k]},k)$ for $k=w-1$. If $w=1$, by considering the $(1,0^{[w-1]},z,2)$ part, we know that $apply(S)$ will always cover $(1,0^{[k]},k)$ for $k=1$.

Therefore, we arrive at a contradiction for all possibilities.

\medskip\noindent\textbf{Case 2: }{$x=1$ and $y\ge 2$.}

Note that $apply(RL^{[y]})=(1,y,-1,0^{[y]},2)$, which covers $(1,k,0^{[k+1]},2)$ for $k=y$. Therefore, there exists some $R$ that appears after those $y$ $L$'s.

If there are no more $L$'s afterwards, then $S$ is in the form $RL^{[y]}R^{[z]}$ for some $z\ge 1$. In that case, we have $apply(S)=(1,y,0^{[z]},-1,z,0^{[y-1]},2)$. If $y\ge z+1$, then the $(y,0^{[z]},-1,z)$ part covers $(1,0^{[k]},k)$ for $k=z+1$ (recall the ``$max(x_i,0)\ge y_{l_2+1-i}$'' part in the definition of covering). If $z+1\ge y$, then the $(z,0^{[y-1]},2)$ part covers $(1,0^{[k]},k)$ for $k=y-1$ (recall that $y\ge 2$ here). So, $S$ is in the form $RL^{[y]}R^{[z]}L\dots$ for some $z\ge 1$.

Note that $apply(RL^{[y]}R^{[z]}L)=(1,y,0^{[z-1]},1,-1,0,z,0^{[y-1]},2)$.

Since the $-1$ is after the $(y,0^{[z-1]},1)$ part, any further deep blowups cannot affect this part except potentially increasing the value of the final $1$. Hence, if $y\ge z-1$ and $z\ge 2$, $apply(S)$ will always cover $(1,0^{[k]},k)$ for $k=z-1$. If $z=1$, the $(1,y,0^{[z-1]},1)$ part covers $(1,0^{[k]},k)$ for $k=1$.

Similarly, any further deep blowups cannot affect the $(z,0^{[y-1]},2)$ part. Therefore, if $z+1\ge y$, $apply(S)$ will always cover $(1,0^{[k]},k)$ for $k=y-1$ (recall that $y\ge 2$ here).

Therefore, we arrive at a contradiction for all possibilities.

\medskip\noindent\textbf{Case 3: }{$x\ge 2$ and $y=1$.}

Note that $apply(R^{[x]}L)=(1,0^{[x-1]},1,-1,0,x+1)$, which covers $(1,0^{[k]},k)$ for $k=2$. Therefore, there exists some $R$ that appears after the $L$.

If there are no more $L$'s afterwards, then $S$ is in the form $R^{[x]}LR^{[z]}$ for some $z\ge 1$. In that case, we have $apply(S)=(1,0^{[x-1]},1,0^{[z]},-1,z,x+1)$, which covers $(1,0^{[k_1]},1,0^{[k_2+1]},k_2,k_1+2)$ for $(k_1,k_2)=(x-1,z-1)$.

So, $S$ is in the form $R^{[x]}LR^{[z]}L\dots$.

Let $w\ge 1$ be the number of consecutive $L$'s that follows $R^{[x]}LR^{[z]}$.

Note that $apply(R^{[x]}LR^{[z]}L^{[w]})=(1,0^{[x-1]},1,0^{[z-1]},w,-1,0^{[w]},z,x+1)$.

Since the $-1$ is behind the $w$, any subsequent deep blowups cannot affect the $(1,0^{[z-1]},w)$ part except for potentially increasing the value of the entry that currently has value $w$. If $w\ge z-1$, then $S$ always covers $(1,0^{[k]},k)$ for $k=z-1$. Therefore, we have $w+1<z$.

Hence, by considering $w,-1,0^{[w]},z$ part, we know that $apply(R^{[x]}LR^{[z]}L^{[w]})$ covers $(1,0^{[k]},k)$ for $k=w+1$. Therefore, $S$ is not $R^{[x]}LR^{[z]}L^{[w]}$. So, $S=R^{[x]}LR^{[z]}L^{[w]}R\dots$.

Note that $apply(R^{[x]}LR^{[z]}L^{[w]}R)=(1,0^{[x-1]},1,0^{[z-1]},w,0,-1,1,0^{[w-1]},z,x+1)$.

Since the $(1,0^{[w-1]},z,x+1)$ part is behind the $-1$, any subsequent deep blowups cannot affect this part except for potentially increasing the value of the entry that currently has value $1$. Since $w+1<z$, we have $w-1<z$. Hence, if $w\ge 2$, we arrive at a contradiction by noting that $S$ covers $(1,0^{[k]},k)$ for $k=w-1$. If $w=1$, then this part becomes $(1,z,x+1)$, which covers $(1,0^{[k]},k)$ for $k=1$.
\end{proof}

\begin{lem}\label{lem:n covering}
Let $\textbf{n}$ be obtained from $(0)$ through a sequence of strict blowup, with at least one of those strict blowups being a middle blowup. Then, $adjusted(\textbf{n})$ covers at least one of the following:
\begin{enumerate}
\item $(1,0^{[k]},k)$ for some $k\ge 1$;
\item $(1,k,0^{[k+1]},2)$ for some $k\ge 0$;
\item $(1,0^{[k_1]},1,0^{[k_2+1]},k_2,k_1+2)$ for some $k_1,k_2\ge 0$.
\end{enumerate}
\end{lem}

\begin{proof}
Consider the rearranged blowup sequence given by Lemma~\ref{lem:standard order}. The original blowup sequence contains a middle blowup. In the rearrangement of Lemma~\ref{lem:standard order}, the order of blowups is changed only by the interchanges in Lemmas~\ref{lem:end blowups first}, \ref{lem:deep swap}, and \ref{lem:shallow swap}. These interchanges preserve whether a blowup is an end blowup or a middle blowup. Hence, the reordered sequence also contains a middle blowup.

The first middle blowup in the reordered sequence is a shallow blowup, because troughs are created only by shallow blowups, whereas a deep blowup can occur only next to an existing trough. Consider the part of the reordered sequence beginning with its final shallow blowup.

In terms of adjusted values, the shallow blowup gives
$$(\dots,x,y,\dots)\rightarrow(\dots,x+1,-1,y+1,\dots)$$
for some $x,y\ge 0$.

By the normal form in Lemma~\ref{lem:standard order}, every subsequent blowup is a deep blowup next to the trough created by this shallow blowup. Therefore, $adjusted(\textbf{n})$ is in the form
$$(\dots,x+a_1,a_2,\dots,a_{i-1},y+a_i,\dots)$$
for some $i$, where $(a_1,\dots,a_i)=apply(S)$, where $S$ is the left-right sequence associated with the remaining deep blowups.

The result follows from Lemma~\ref{lem:LR classification}.
\end{proof}

\begin{lem}\label{lem:1 not hold}
Let $L(p,q)$ be a lens space, and $[b_1,\dots,b_l]^-$ be the Hirzebruch-Jung expansion of $p\slash(p-q)$. If Theorem~\ref{thm:main} statement~\ref{it:min_filling} does not hold for $L(p,q)$, then $adjusted((b_1,\dots,b_l))$ covers at least one of the following:
\begin{enumerate}
\item $(1,0^{[k]},k)$ for some $k\ge 1$;
\item $(1,k,0^{[k+1]},2)$ for some $k\ge 0$;
\item $(1,0^{[k_1]},1,0^{[k_2+1]},k_2,k_1+2)$ for some $k_1,k_2\ge 0$.
\end{enumerate}
\end{lem}

\begin{proof}
Suppose Theorem~\ref{thm:main} statement~\ref{it:min_filling} does not hold for $L(p,q)$.

Lemma~\ref{lem:thm equiv 1} implies that there is some $\textbf{n}\in Z_{p,q}$ such that a sequence of strict blowups that produces $\textbf{n}$ from $(0)$ involves at least one middle blowup. By Lemma~\ref{lem:n covering}, we know that $adjusted(\textbf{n})$ covers at least one item from that list. The definition of $Z_{p,q}$ implies that $(b_1,\dots,b_l)$ covers $\textbf{n}$. Therefore, $adjusted((b_1,\dots,b_l))$ covers $adjusted(\textbf{n})$. The result follows from Lemma~\ref{lem:cover transitive}.
\end{proof}

\begin{lem}\label{lem:3 consec}
Let $L=L(p,q)$ be a lens space, and $[b_1,\dots,b_l]^-$ be the Hirzebruch-Jung expansion of $p\slash(p-q)$. If $adjusted((b_1,\dots,b_l))$ consists of three consecutive non-zero entries, then $L$ contains at least one configuration in $B$ as an induced subgraph.
\end{lem}

\begin{proof}
Let $(a_1,a_2,a_3)$ be consecutive entries in $adjusted((b_1,\dots,b_l))$ such that $a_1,a_2,a_3>0$. By Lemma~\ref{lem:dual}, the canonical negative-definite linear plumbing graph associated to $L$ contains
$$\begin{tikzpicture}[xscale=1.0,yscale=1,baseline={(0,0)}]
    \node at (-2.1,0.4) {$-2$};
		\node at (-0.1,0.4) {$-2$};
    \node at (0.9,0.4) {$-3$};
    \node at (1.9,0.4) {$-2$};
		\node at (3.9,0.4) {$-2$};
		\node at (4.9,0.4) {$-3$};
		\node at (5.9,0.4) {$-2$};
		\node at (7.9,0.4) {$-2$};
    \node (A-2) at (-2,0) {$\bullet$};
    \node (A-1) at (-1,0) {$\cdots$};
    \node (A0) at (0,0) {$\bullet$};
    \node (A1) at (1,0) {$\bullet$};
    \node (A2) at (2,0) {$\bullet$};
		\node (A3) at (3,0) {$\cdots$};
		\node (A4) at (4,0) {$\bullet$};
		\node (A5) at (5,0) {$\bullet$};
		\node (A6) at (6,0) {$\bullet$};
		\node (A7) at (7,0) {$\cdots$};
		\node (A8) at (8,0) {$\bullet$};
    \path (A-2) edge [-] node [auto] {$\scriptstyle{}$} (A-1);
		\path (A-1) edge [-] node [auto] {$\scriptstyle{}$} (A0);
		\path (A0) edge [-] node [auto] {$\scriptstyle{}$} (A1);
    \path (A1) edge [-] node [auto] {$\scriptstyle{}$} (A2);
		\path (A2) edge [-] node [auto] {$\scriptstyle{}$} (A3);
		\path (A3) edge [-] node [auto] {$\scriptstyle{}$} (A4);
		\path (A4) edge [-] node [auto] {$\scriptstyle{}$} (A5);
		\path (A5) edge [-] node [auto] {$\scriptstyle{}$} (A6);
		\path (A6) edge [-] node [auto] {$\scriptstyle{}$} (A7);
		\path (A7) edge [-] node [auto] {$\scriptstyle{}$} (A8);
		\draw [decorate,decoration={brace,amplitude=5pt,mirror,raise=2ex}] (-2,0) -- (0,0) node[midway,yshift=-2em]{$a_1-1$};
		\draw [decorate,decoration={brace,amplitude=5pt,mirror,raise=2ex}] (2,0) -- (4,0) node[midway,yshift=-2em]{$a_2-1$};
		\draw [decorate,decoration={brace,amplitude=5pt,mirror,raise=2ex}] (6,0) -- (8,0) node[midway,yshift=-2em]{$a_3-1$};
  \end{tikzpicture}$$
as an induced subgraph. This contains $B_{a_2-1}$ as an induced subgraph.
\end{proof}

\begin{lem}\label{lem:case 1}
Let $L=L(p,q)$ be a lens space, and $[b_1,\dots,b_l]^-$ be the Hirzebruch-Jung expansion of $p\slash(p-q)$. If $adjusted((b_1,\dots,b_l))$ covers
$$(1,0^{[k]},k)$$
for some $k\ge 1$, then the canonical negative-definite linear plumbing graph associated to $L$ contains at least one configuration in $A\cup B\cup C$ as an induced subgraph.
\end{lem}

\begin{proof}
Since $b_1,\dots,b_l$ came from the Hirzebruch-Jung expansion, their adjusted values are all non-negative. So, $adjusted((b_1,\dots,b_l))$ contains a block of entries (possibly in reverse order) $(a_0,\dots,a_{k+1})$ such that $a_0\ge 1$, $a_{k+1}\ge k$, and $a_i\ge 0$ for all $i$.

Let $u$ be the number of consecutive $0$'s before $a_{k+1}$ in $(a_0,\dots,a_{k+1})$. (we set $u=0$ if $a_k$ is not $0$)

\medskip\noindent\textbf{Case 1: }{$u\ge 1$.}

According to Lemma~\ref{lem:dual}, and since $a_{k+1}\ge k$, we know that the canonical negative-definite linear plumbing graph associated to $L$ contains
$$\begin{tikzpicture}[xscale=1.0,yscale=1,baseline={(0,0)}]
    \node at (0.9,0.4) {$-(3+u)$};
    \node at (1.9,0.4) {$-2$};
		\node at (3.9,0.4) {$-2$};
    \node (A1) at (1,0) {$\bullet$};
    \node (A2) at (2,0) {$\bullet$};
		\node (A3) at (3,0) {$\cdots$};
		\node (A4) at (4,0) {$\bullet$};
    \path (A1) edge [-] node [auto] {$\scriptstyle{}$} (A2);
		\path (A2) edge [-] node [auto] {$\scriptstyle{}$} (A3);
		\path (A3) edge [-] node [auto] {$\scriptstyle{}$} (A4);
		\draw [decorate,decoration={brace,amplitude=5pt,mirror,raise=2ex}] (2,0) -- (4,0) node[midway,yshift=-2em]{$k-1$};
  \end{tikzpicture}$$
as an induced subgraph. Since $a_0\ge 1$, we know that $u\le k$, and therefore $k-1\ge u-1$. Hence, this contains $A_{u-1}$ as an induced subgraph.

\medskip\noindent\textbf{Case 2: }{$u=0$.}

In that case, $a_k>0$. Let $v$ be the number of consecutive $0$'s before $a_k$. If $v=0$, then Lemma~\ref{lem:3 consec} implies that the canonical negative-definite linear plumbing graph associated to $L$ contains at least one configuration in $B$ as an induced subgraph. If $v=1$, then Lemma~\ref{lem:dual} implies that the canonical negative-definite linear plumbing graph has a vertex with weight $-4$, and therefore contains $A_0$ as an induced subgraph. If $v\ge 2$, then the canonical negative-definite linear plumbing graph associated to $L$ contains
$$\begin{tikzpicture}[xscale=1.0,yscale=1,baseline={(0,0)}]
    \node at (0.9,0.4) {$-(3+v)$};
    \node at (1.9,0.4) {$-2$};
		\node at (3.9,0.4) {$-2$};
		\node at (4.9,0.4) {$-3$};
		\node at (5.9,0.4) {$-2$};
		\node at (7.9,0.4) {$-2$};
    \node (A1) at (1,0) {$\bullet$};
    \node (A2) at (2,0) {$\bullet$};
		\node (A3) at (3,0) {$\cdots$};
		\node (A4) at (4,0) {$\bullet$};
		\node (A5) at (5,0) {$\bullet$};
		\node (A6) at (6,0) {$\bullet$};
		\node (A7) at (7,0) {$\cdots$};
		\node (A8) at (8,0) {$\bullet$};
    \path (A1) edge [-] node [auto] {$\scriptstyle{}$} (A2);
		\path (A2) edge [-] node [auto] {$\scriptstyle{}$} (A3);
		\path (A3) edge [-] node [auto] {$\scriptstyle{}$} (A4);
		\path (A4) edge [-] node [auto] {$\scriptstyle{}$} (A5);
		\path (A5) edge [-] node [auto] {$\scriptstyle{}$} (A6);
		\path (A6) edge [-] node [auto] {$\scriptstyle{}$} (A7);
		\path (A7) edge [-] node [auto] {$\scriptstyle{}$} (A8);
		\draw [decorate,decoration={brace,amplitude=5pt,mirror,raise=2ex}] (2,0) -- (4,0) node[midway,yshift=-2em]{$a_k-1$};
		\draw [decorate,decoration={brace,amplitude=5pt,mirror,raise=2ex}] (6,0) -- (8,0) node[midway,yshift=-2em]{$k-1$};
  \end{tikzpicture}$$
as an induced subgraph. Since $a_0\ge 1$, we know that $v\le k-1$, and therefore $k-1\ge v-2+2$. Hence, this contains $C_{v-2,a_k-1}$ as an induced subgraph.
\end{proof}

\begin{lem}\label{lem:case 2}
Let $L=L(p,q)$ be a lens space, and $[b_1,\dots,b_l]^-$ be the Hirzebruch-Jung expansion of $p\slash(p-q)$. If $adjusted((b_1,\dots,b_l))$ covers
$$(1,k,0^{[k+1]},2)$$
for some $k\ge 0$, then the canonical negative-definite linear plumbing graph associated to $L$ contains at least one configuration in $A\cup B\cup C\cup D$ as an induced subgraph.
\end{lem}

\begin{proof}
If $k=0$, this corresponds to the $k=2$ case in Lemma~\ref{lem:case 1}. So, we let $k\ge 1$.

Let $(a_0,\dots,a_{k+3})$ be a block of entries in $adjusted((b_1,\dots,b_l))$ (possibly in reverse order) such that $a_0\ge 1$, $a_1\ge k$, and $a_{k+3}\ge 2$. Let $u$ be the number of consecutive $0$'s after $a_1$.

If $u=0$, then we can use Lemma~\ref{lem:3 consec}. If $1 \le u<k+1$, then this covers $(1,0^{[k']},k')$ for $k'=u$, and therefore we reduce it to Lemma~\ref{lem:case 1} by transitivity of covering.
Hence, the only case we need to check is when $a_2=\cdots=a_{k+2}=0$.

By Lemma~\ref{lem:dual}, the canonical negative-definite linear plumbing graph associated to $L$ contains

$$\begin{tikzpicture}[xscale=1.0,yscale=1,baseline={(0,0)}]
    \node at (0.9,0.4) {$-3$};
    \node at (1.9,0.4) {$-2$};
		\node at (3.9,0.4) {$-2$};
		\node at (4.9,0.4) {$-(k+4)$};
		\node at (5.9,0.4) {$-2$};
    \node (A1) at (1,0) {$\bullet$};
    \node (A2) at (2,0) {$\bullet$};
		\node (A3) at (3,0) {$\cdots$};
		\node (A4) at (4,0) {$\bullet$};
		\node (A5) at (5,0) {$\bullet$};
		\node (A6) at (6,0) {$\bullet$};
    \path (A1) edge [-] node [auto] {$\scriptstyle{}$} (A2);
		\path (A2) edge [-] node [auto] {$\scriptstyle{}$} (A3);
		\path (A3) edge [-] node [auto] {$\scriptstyle{}$} (A4);
		\path (A4) edge [-] node [auto] {$\scriptstyle{}$} (A5);
		\path (A5) edge [-] node [auto] {$\scriptstyle{}$} (A6);
		\draw [decorate,decoration={brace,amplitude=5pt,mirror,raise=2ex}] (2,0) -- (4,0) node[midway,yshift=-2em]{$a_1-1$};
  \end{tikzpicture}$$
as an induced subgraph. If $a_1\ge k+1$, this contains $A_k$. If $a_1=k$, this is $D_{k-1,0}$.
\end{proof}

\begin{lem}\label{lem:case 3}
Let $L=L(p,q)$ be a lens space, and $[b_1,\dots,b_l]^-$ be the Hirzebruch-Jung expansion of $p\slash(p-q)$. If $adjusted((b_1,\dots,b_l))$ covers
$$(1,0^{[k_1]},1,0^{[k_2+1]},k_2,k_1+2)$$
for some $k_1,k_2\ge 0$, then the canonical negative-definite linear plumbing graph associated to $L$ contains at least one configuration in $A\cup B\cup C\cup D$ as an induced subgraph.
\end{lem}

\begin{proof}
If $k_2=0$, then this covers $(1,0^{[k]},k)$ for $k=2$, so we reduce it to Lemma~\ref{lem:case 1} by transitivity of covering. Hence, we assume that $k_2\ge 1$.

Let $(a_0,\dots,a_{k_1+k_2+4})$ be a block of entries in $adjusted((b_1,\dots,b_l))$ (possibly in reverse order) such that $a_0\ge 1$, $a_{k_1+1}\ge 1$, $a_{k_1+k_2+3}\ge k_2$, and $a_{k_1+k_2+4}\ge k_1+2$. Let $u$ be the number of consecutive $0$'s before $a_{k_1+k_2+3}$.

If $u=0$, then we can use Lemma~\ref{lem:3 consec}. If $1\le u\le k_2$, then this covers $(1,0^{[k]},k)$ for $k=u$, and therefore we reduce it to Lemma~\ref{lem:case 1} by transitivity of covering.

Hence, the only case we need to check is when $a_{k_1+2}=\cdots=a_{k_1+k_2+2}=0$.

If $a_{k_1+1}\ge 2$, then this covers $(1,k,0^{[k+1]},2)$ for $k=k_2$, and therefore we reduce it to Lemma~\ref{lem:case 2}.

Hence, we can assume that $a_{k_1+1}=1$.

Let $v$ be the number of consecutive $0$'s before $a_{k_1+1}$. By Lemma~\ref{lem:dual}, the canonical negative-definite linear plumbing graph associated to $L$ contains

$$\begin{tikzpicture}[xscale=1.0,yscale=1,baseline={(0,0)}]
    \node at (0.9,0.4) {$-(3+v)$};
    \node at (1.9,-0.4) {$-(4+k_2)$};
		\node at (2.9,0.4) {$-2$};
		\node at (4.9,0.4) {$-2$};
		\node at (5.9,0.4) {$-3$};
		\node at (6.9,0.4) {$-2$};
		\node at (8.9,0.4) {$-2$};
    \node (A1) at (1,0) {$\bullet$};
    \node (A2) at (2,0) {$\bullet$};
		\node (A3) at (3,0) {$\bullet$};
		\node (A4) at (4,0) {$\cdots$};
		\node (A5) at (5,0) {$\bullet$};
		\node (A6) at (6,0) {$\bullet$};
		\node (A7) at (7,0) {$\bullet$};
		\node (A8) at (8,0) {$\cdots$};
		\node (A9) at (9,0) {$\bullet$};
    \path (A1) edge [-] node [auto] {$\scriptstyle{}$} (A2);
		\path (A2) edge [-] node [auto] {$\scriptstyle{}$} (A3);
		\path (A3) edge [-] node [auto] {$\scriptstyle{}$} (A4);
		\path (A4) edge [-] node [auto] {$\scriptstyle{}$} (A5);
		\path (A5) edge [-] node [auto] {$\scriptstyle{}$} (A6);
		\path (A6) edge [-] node [auto] {$\scriptstyle{}$} (A7);
		\path (A7) edge [-] node [auto] {$\scriptstyle{}$} (A8);
		\path (A8) edge [-] node [auto] {$\scriptstyle{}$} (A9);
		\draw [decorate,decoration={brace,amplitude=5pt,mirror,raise=2ex}] (3,0) -- (5,0) node[midway,yshift=-2em]{$k_2-1$};
		\draw [decorate,decoration={brace,amplitude=5pt,mirror,raise=2ex}] (7,0) -- (9,0) node[midway,yshift=-2em]{$k_1+1$};
		\end{tikzpicture}$$
as an induced subgraph. Since $a_0\ge 1$, we have $v\le k_1$, and hence $v+1\le k_1+1$. Therefore, this contains $D_{k_2-1,v+1}$.
\end{proof}

\begin{lem}\label{lem:3 implies 1}
Let $L=L(p,q)$ be a lens space. Then, in Theorem~\ref{thm:main}, \ref{it:combinatorial} implies \ref{it:min_filling}.
\end{lem}

\begin{proof}
In Theorem~\ref{thm:main}, if \ref{it:min_filling} is false, then Lemmas~\ref{lem:1 not hold}, \ref{lem:case 1}, \ref{lem:case 2}, and \ref{lem:case 3} imply that the canonical negative-definite linear plumbing graph associated to $L$ contains at least one configuration in $A\cup B\cup C\cup D$ as an induced subgraph. By Remark~\ref{rem:thm}, this is equivalent to the canonical negative-definite linear plumbing graph associated to $L$ contains at least one configuration in $A\cup B\cup C'\cup D'$ as an induced subgraph.
\end{proof}

It remains to prove \ref{it:min_filling}$\Rightarrow$\ref{it:combinatorial}.

\begin{lem}\label{lem:before 1 implies 3}
Let $L=L(p,q)$ be a lens space, and $[b_1,\dots,b_l]^-$ be the Hirzebruch-Jung expansion of $p\slash(p-q)$. If the canonical negative-definite linear plumbing graph associated to $L$ contains a configuration in $A\cup B\cup C\cup D$ as an induced subgraph, then $adjusted((b_1,\dots,b_l))$ covers at least one of
\begin{enumerate}
\item $(1,0^{[k]},k)$ for some $k\ge 1$;
\item $(1,k,0^{[k+1]},2)$ for some $k\ge 0$;
\item $(1,0^{[k_1]},1,0^{[k_2+1]},k_2,k_1+2)$ for some $k_1,k_2\ge 0$.
\end{enumerate}
\end{lem}

\begin{proof}
Let $k\ge 0$.

If $L$ contains $A_k$, then by Lemma~\ref{lem:dual}, $adjusted((b_1,\dots,b_l))$ contains $(a_1,0^{[k+1]},a_2)$ (possibly in reverse order) for some $a_1\ge 1$ and $a_2\ge k+1$, which covers $(1,0^{[k']},k')$ for $k'=k+1$.

If $L$ contains $B_k$, then by Lemma~\ref{lem:dual}, $adjusted((b_1,\dots,b_l))$ contains $(a_1,a_2,a_3)$ (possibly in reverse order) for some $a_1,a_2,a_3\ge 1$, which covers $(1,0^{[k']},k')$ for $k'=1$.

Let $s\ge 0$.

If $L$ contains $C_{k,s}$, then by Lemma~\ref{lem:dual}, $adjusted((b_1,\dots,b_l))$ contains $(a_1,0^{[k+2]},s+1,a_2)$ for some $a_1\ge 1$ and $a_2\ge k+3$, which covers $(1,0^{[k']},k')$ for $k'=k+3$.

For $D_{k,s}$, we separate it into the $s=0$ case and the $s\ge 1$ case.

If $L$ contains $D_{k,0}$, then by Lemma~\ref{lem:dual}, $adjusted((b_1,\dots,b_l))$ contains $(a_1,0^{[k+2]},k+1,a_2)$ for some $a_1\ge 2$ and $a_2\ge 1$, which covers $(1,k',0^{[k'+1]},2)$ for $k'=k+1$.

If $L$ contains $D_{k,s}$ with $s\ge 1$, then by Lemma~\ref{lem:dual}, $adjusted((b_1,\dots,b_l))$ contains $(a_1,0^{[s-1]},1,0^{[k+2]},k+1,a_2)$ for some $a_1\ge 1$ and $a_2\ge s+1$, which covers $(1,0^{[k_1]},1,0^{[k_2+1]},k_2,k_1+2)$ for $(k_1,k_2)=(s-1,k+1)$.
\end{proof}

\begin{lem}\label{lem:1 implies 3}
Let $L=L(p,q)$ be a lens space. Then, in Theorem~\ref{thm:main}, \ref{it:min_filling} implies \ref{it:combinatorial}.
\end{lem}

\begin{proof}
Suppose \ref{it:combinatorial} does not hold. By Remark~\ref{rem:thm}, this is equivalent to the canonical negative-definite linear plumbing graph associated to $L$ contains a configuration in $A\cup B\cup C\cup D$ as an induced subgraph. We want to show that \ref{it:min_filling} does not hold. By Lemma~\ref{lem:thm equiv 1}, it suffices to show that there exists some $\textbf{n}\in Z_{p,q}$ obtained from a strict blowup process from $(0)$ that involves at least one middle blowup.

By Lemma~\ref{lem:before 1 implies 3}, $adjusted((b_1,\dots,b_l))$ covers at least one of
\begin{enumerate}
\item $(1,0^{[k]},k)$ for some $k\ge 1$;
\item $(1,k,0^{[k+1]},2)$ for some $k\ge 0$;
\item $(1,0^{[k_1]},1,0^{[k_2+1]},k_2,k_1+2)$ for some $k_1,k_2\ge 0$.
\end{enumerate}

By symmetry, without loss of generality, assume that the covering uses ``$max(x_{j+i},0)\ge y_i$'' part of the definition. Let $l$ be as in $(b_1,\dots,b_l)$, and $j$ be as in the definition of covering.

Note that
\begin{align*}
apply(R^{[x]})&=(1,0^{[x]},-1,x+1) \\
apply(RL^{[x]})&=(1,x,-1,0^{[x]},2) \\
apply(R^{[x]}LR^{[y]})&=(1,0^{[x-1]},1,0^{[y]},-1,y,x+1)
\end{align*}

If $adjusted((b_1,\dots,b_l))$ covers $(1,0^{[k]},k)$, we can create such $\textbf{n}$ by doing $(l-k-1)$ end blowups, then do a shallow blowup in the $(j+1)^\text{th}$ gap, then do perform deep blowups on the trough according to the sequence $R^{[k-1]}$.

If $adjusted((b_1,\dots,b_l))$ covers $(1,k,0^{[k+1]},2)$, we can create such $\textbf{n}$ by doing $(l-k-3)$ end blowups, then do a shallow blowup in the $(j+1)^\text{th}$ gap, then do perform deep blowups on the trough according to the sequence $RL^{[k]}$.

If $adjusted((b_1,\dots,b_l))$ covers $(1,0^{[k_1]},1,0^{[k_2+1]},k_2,k_1+2)$, we can create such $\textbf{n}$ by doing $(l-k_1-k_2-4)$ end blowups, then do a shallow blowup in the $(j+1)^\text{th}$ gap, then do perform deep blowups on the trough according to the sequence $R^{[k_1+1]}LR^{[k_2]}$.
\end{proof}

Theorem~\ref{thm:main} is proved by combining Corollary~\ref{cor:1 2 equiv}, Lemma~\ref{lem:3 implies 1}, and Lemma~\ref{lem:1 implies 3}.

\begin{rem}\label{rem:final}
Instead of doing Lemmas~\ref{lem:before 1 implies 3} and \ref{lem:1 implies 3}, the implication \ref{it:min_filling}$\Rightarrow$\ref{it:combinatorial} can alternatively be
argued geometrically. Suppose the canonical plumbing graph contains a configuration in $\mathcal A\cup\mathcal B\cup\mathcal C'\cup\mathcal D'$ as an induced subgraph. A symplectic rational blowdown, described below, produces a symplectic filling with smaller $b_2$.

The graph $A_k$ in this paper is the same as $C_{k+2}$ in \cite{Symington1998}. For $B_k$, symplectically resolving\footnote{See \cite[Lemma 2.6]{symplectichats} for an explicit construction of the symplectic resolution of positive intersection points.} all intersection points in the chain produces a symplectic $(-4)$-sphere. For $C_{k,s}$, symplectically resolving the $s+1$ intersection points in the chain of spheres between the $(-(5+k))$-sphere
and the $(-3)$-sphere produces a symplectic $(-(k+6))$-sphere. Together with the $k+2$ adjacent
$(-2)$-spheres, it forms $C_{k+4}$ in \cite{Symington1998}. Finally, a continued-fraction calculation identifies $D_{k,s}$ with $C_{ks+2k+3s+5,k+3}$ in \cite{Symington2001}.
\end{rem}

Finally, we prove Corollary~\ref{cor:infinite}.

\begin{proof}[Proof of Corollary~\ref{cor:infinite}]
Suppose there is another list of configurations $L$ such that not containing any configurations from the list as an induced subgraph is equivalent to the canonical negative-definite plumbing having minimum $b_2$ across all symplectic fillings. We argue that every configuration in $A\cup B\cup C'\cup D'$ must be in $L$.

Let $G$ be a configuration in $A\cup B\cup C'\cup D'$. Then, since $G$ contains itself as an induced subgraph, by Theorem~\ref{thm:main}, the lens space corresponding to $G$ has the property that the canonical negative-definite plumbing does not have minimum $b_2$ across all symplectic fillings. Therefore, $G$ contains some configuration $G_L$ in $L$ as an induced subgraph.

Since $G$ is a linear graph (i.e. path graph), $G_L$ must be a disjoint union of (possibly one) linear graphs. Let $k$ be the number of connected components of $G_L$, and let the $i^{th}$ component be in the form
$$\begin{tikzpicture}[xscale=1.0,yscale=1,baseline={(0,0)}]
    \node at (0.9,0.4) {$-a_{i,1}$};
    \node at (1.9,0.4) {$-a_{i,2}$};
		\node at (3.9,0.4) {$-a_{i,k_i}$};
    \node (A1) at (1,0) {$\bullet$};
    \node (A2) at (2,0) {$\bullet$};
		\node (A3) at (3,0) {$\cdots$};
		\node (A4) at (4,0) {$\bullet$};
    \path (A1) edge [-] node [auto] {$\scriptstyle{}$} (A2);
		\path (A2) edge [-] node [auto] {$\scriptstyle{}$} (A3);
		\path (A3) edge [-] node [auto] {$\scriptstyle{}$} (A4);
  \end{tikzpicture}$$
for each $1\le i\le k$. Let $x$ be an integer much bigger than the total number of vertices $G_L$ has. Consider the linear graph formed by connecting every $i^{th}$ and $(i+1)^{th}$ components together through an added vertex of weight $-x$:
$$\begin{tikzpicture}[xscale=1.0,yscale=1,baseline={(0,0)}]
    \node at (0.9,0.4) {$-a_{1,1}$};
		\node at (2.9,0.4) {$-a_{1,k_1}$};
		\node at (3.9,0.4) {$-x$};
    \node at (4.9,0.4) {$-a_{2,1}$};
		\node at (6.9,0.4) {$-a_{2,k_2}$};
		\node at (7.9,0.4) {$-x$};
		\node at (9.9,0.4) {$-x$};
    \node at (10.9,0.4) {$-a_{k,1}$};
		\node at (12.9,0.4) {$-a_{k,k_k}$};
    \node (A1) at (1,0) {$\bullet$};
    \node (A2) at (2,0) {$\cdots$};
		\node (A3) at (3,0) {$\bullet$};
		\node (A4) at (4,0) {$\bullet$};
		\node (A5) at (5,0) {$\bullet$};
    \node (A6) at (6,0) {$\cdots$};
		\node (A7) at (7,0) {$\bullet$};
		\node (A8) at (8,0) {$\bullet$};
		\node (A9) at (9,0) {$\cdots$};
		\node (A10) at (10,0) {$\bullet$};
    \node (A11) at (11,0) {$\bullet$};
		\node (A12) at (12,0) {$\cdots$};
    \node (A13) at (13,0) {$\bullet$};
    \path (A1) edge [-] node [auto] {$\scriptstyle{}$} (A2);
		\path (A2) edge [-] node [auto] {$\scriptstyle{}$} (A3);
		\path (A3) edge [-] node [auto] {$\scriptstyle{}$} (A4);
    \path (A4) edge [-] node [auto] {$\scriptstyle{}$} (A5);
		\path (A5) edge [-] node [auto] {$\scriptstyle{}$} (A6);
		\path (A6) edge [-] node [auto] {$\scriptstyle{}$} (A7);
    \path (A7) edge [-] node [auto] {$\scriptstyle{}$} (A8);
		\path (A8) edge [-] node [auto] {$\scriptstyle{}$} (A9);
		\path (A9) edge [-] node [auto] {$\scriptstyle{}$} (A10);
    \path (A10) edge [-] node [auto] {$\scriptstyle{}$} (A11);
    \path (A11) edge [-] node [auto] {$\scriptstyle{}$} (A12);
    \path (A12) edge [-] node [auto] {$\scriptstyle{}$} (A13);
  \end{tikzpicture}.$$

Since this contains $G_L$ as an induced subgraph, the lens space corresponding to this linear graph has the property that the canonical negative-definite plumbing does not have minimum $b_2$ across all symplectic fillings. Therefore, it contains some configuration $G'$ in $A\cup B\cup C'\cup D'$ as an induced subgraph. Observe that in every configuration in $A\cup B\cup C'\cup D'$, the biggest value of the absolute value of a vertex's weight is at most 3 more than the number of vertices. Since taking subgraphs does not increase the number of vertices, by choosing a big enough $x$, we know that $G'$ cannot contain a vertex of weight $-x$. Hence, $G'$ is an induced subgraph of a component of $G_L$. Therefore, $G'$ is an induced subgraph of $G_L$.

Since $G'$ is an induced subgraph of $G_L$, and $G_L$ is an induced subgraph of $G$, it follows that $G'$ is an induced subgraph of $G$. Since the list $A\cup B\cup C'\cup D'$ is minimal, we must have $G=G'$. Hence, we have $G=G_L=G'$. Therefore, $G=G_L\in L$.
\end{proof}

\bibliographystyle{alpha}
\def\MR#1{}
\bibliography{bib}

\end{document}